\documentclass[12pt]{l4dc2020} 

\usepackage{float}

\newcommand{\xresp}{\Phi_x}
\newcommand{\uresp}{\Phi_u}
\newcommand{\xrespzt}{\mathbf{\Phi_x}}
\newcommand{\urespzt}{\mathbf{\Phi_u}}
\newcommand{\xzt}{\textbf{x}}
\newcommand{\uzt}{\textbf{u}}
\newcommand{\wzt}{\textbf{w}}

\newcommand{\Deltazt}{\mathbf{\Delta}}

\newcommand{\xfirzt}{\mathbf{\Phi_x}}
\newcommand{\ufirzt}{\mathbf{\Phi_u}}

\newcommand{\xfirstk}{\mathbf{\bar{\Phi}_x}}
\newcommand{\ufirstk}{\mathbf{\bar{\Phi}_u}}

\newcommand{\xuresp}{\mathbf{\bar{\Phi}}}

\newcommand{\firl}{F}

\newcommand{\dumtf}{H}
\newcommand{\dumtfzt}{\textbf{H}}

\newcommand{\cert}{P}

\newcommand{\htwocost}{J_{\htwo}}

\newcommand{\kzt}{\textbf{K}}

\newcommand{\gpoli}{\genpolicy_1}
\newcommand{\gpolii}{\genpolicy_2}

\newcommand{\timexp}{{T_e}}

\newcommand{\cloop}{\mathcal{S}_\textup{CL}}

\newcommand{\mset}{\matrixmodelset}

\newcommand{\bigo}{\mathcal{O}}

\newcommand{\timeh}{N}

\newcommand{\robustcost}{J_{\infty}}
\newcommand{\robustcostopt}{J_{\infty}^*}

\newcommand{\uncertlin}{\uncert_{\ell}}

\newcommand{\xulp}{\mathbf{\bar{\Phi}}_{\text{nom}}}

\newcommand{\real}{\mathbb{R}}
\newcommand{\sym}[1]{\mathbb{S}^{#1}}

\newcommand{\naturals}{\mathbb{N}}

\newcommand{\kron}{\otimes}

\newcommand{\norm}[2]{\left\Vert#1\right\Vert_{#2}}
\newcommand{\normal}[2]{\mathcal{N}\left({#1},{#2}\right)}

\newcommand{\ev}{\mathbb{E}}

\newcommand{\lbr}{\lbrace}
\newcommand{\rbr}{\rbrace}

\newcommand{\mle}{\preceq}

\newcommand{\mge}{\succeq}

\newcommand{\hinf}{\mathcal{H}_{\infty}}
\newcommand{\htwo}{\mathcal{H}_{2}}
\newcommand{\rhinf}{\mathcal{R}\mathcal{H}_{\infty}}

\newcommand{\data}{\mathcal{D}}
\newcommand{\model}{\mathcal{M}_{\credprob}}
\newcommand{\datainit}{\mathcal{D}_0}

\newcommand{\amodel}{\tilde{\mathcal{M}}}

\newcommand{\matrixmodelset}{{\Theta_m}}

\newcommand{\anom}{\hat{A}}
\newcommand{\bnom}{\hat{B}}
\newcommand{\uncert}{D}

\newcommand{\auncert}{\tilde{\uncert}}

\newcommand{\multiplier}{\lambda}

\newcommand{\trnsp}{^\top}

\newcommand{\lmia}{\mathcal{A}}
\newcommand{\lmib}{\mathcal{B}}
\newcommand{\lmic}{\mathcal{C}}
\newcommand{\lmid}{\mathcal{P}}
\newcommand{\lmif}{\mathcal{F}}
\newcommand{\lmig}{\mathcal{G}}
\newcommand{\lmih}{\mathcal{H}}

\newcommand{\st}{\textup{s.t.}}
\newcommand{\defeq}{:=}
\newcommand{\eqdef}{=:}

\newcommand{\mysec}{\S}

\newcommand{\chiSquare}[2]{\chi^2_{#1}(#2)}

\newcommand{\atr}{{A_\textup{tr}}}
\newcommand{\btr}{{B_\textup{tr}}}

\newcommand{\diststd}{\sigma_w}

\newcommand{\credconst}{c_{\credprob}}

\newcommand{\modelx}{X}

\newcommand{\credprob}{\delta}

\newcommand{\qcost}{Q}
\newcommand{\rcost}{R}
\newcommand{\cost}{c}

\newcommand{\totaltime}{T}

\newcommand{\genpolicy}{\phi}

\newcommand{\policy}{\mathcal{K}}

\title[Optimistic dual control]{Optimistic robust linear quadratic dual control}
\usepackage{times}

\author{%
 \Name{Jack Umenberger} \Email{umnbrgr@mit.edu}\\
 \addr CSAIL, Massachusetts Institute of Technology, Cambridge, MA 02139
 \AND
 \Name{Thomas B. Sch\"on} \Email{thomas.schon@it.uu.se}\\
 \addr Department of Information Technology, Uppsala University, Uppsala, 75236, Sweden%
}

\begin{document}

\maketitle

\vspace{-2em}

\begin{abstract}%
Recent work by \cite{mania2019certainty} has proved that certainty equivalent control achieves nearly optimal regret for linear systems with quadratic costs.
However, when parameter uncertainty is large, certainty equivalence cannot be relied upon to stabilize the true, unknown system.
In this paper, we present a dual control strategy that attempts to combine the performance of certainty equivalence, with the practical utility of robustness.
The formulation preserves structure in the representation of parametric uncertainty, which allows the controller to target reduction of uncertainty in the parameters that `matter most' for the control task, while robustly stabilizing the uncertain system.
Control synthesis proceeds via convex optimization, and the method is illustrated on a numerical example.
\end{abstract}


\section{Introduction}
Since the initial formulation of the `dual control' problem by \cite{feldbaum1960dual} in the 1960s, learning to make decisions in uncertain and dynamic environments has remained a topic of sustained research activity.
However, recent years have witnessed a resurgence of interest in such problems, inspired perhaps in part by the dramatic success of reinforcement learning, cf. \cite{mnih2015human, silver2016mastering}.
Specifically, linear systems with quadratic costs, a.k.a. `the linear quadratic regulator', have been the subject of intense recent study, cf. \cite{matni2019self}.
Such research typically focuses on two main aspects:
i) performance, usually measured in terms of bounds on regret, and
ii) robustness, i.e., stability of the closed-loop system, which is often important in practical applications.
Concerning the former, the work of \cite{mania2019certainty} has proved that `certainty equivalent' (CE) control (nearly) achieves the optimal regret bound; provided that this controller stabilizes the system, which is the case when parameter uncertainty is sufficiently small.
Inspired by this result, the present paper attempts to combine the performance of certainty equivalence with the practical advantages of robustness.
Specifically, we propose a dual control strategy, for linear systems with quadratic costs, that optimizes for performance of the nominal, i.e., most likely, system (as in CE control), while robustly stabilizing the system in the presence of parametric uncertainty.
The dual controller performs `targeted exploration', attempting to reduce uncertainty in the parameters that `matter most' for control, while balancing the exploration-exploitation tradeoff.

The contributions of this paper are twofold.
In \mysec\ref{sec:robust_synthesis}, we present a convex formulation of optimization of quadratic cost, for a nominal linear system, subject to robust stability guarantees under parametric uncertainty. 
This extends the existing system level synthesis (SLS) framework, 
cf. \cite{wang2019system}, 
by preserving structure in the representation of system uncertainty.
In \mysec\ref{sec:dual_synthesis}, we build upon this formulation to present an (approximate) dual control strategy, exploiting the preservation of structure to perform  exploration that targets uncertainty reduction in the specific parameters that are `preventing' certainty equivalent control from stabilizing the uncertain true system.

\paragraph{Related work}
Of greatest relevance to the present paper is the work of \cite{mania2019certainty}, which proves that certainty equivalence, i.e., estimating model parameters via online least squares and then applying LQR, achieves (nearly optimal) $\tilde{\bigo}(\sqrt{T})$ regret.
This result holds when the parameter error is sufficiently small so as to ensure closed-loop stability of the true system, which does not always hold in practice, e.g., \cite{dean2017sample}. Many recent works have addressed the issue of robustness in adaptive control, cf. \cite{dean2017sample,dean2018regret,dean2018safely,cohen2018online}.
The work of \cite{umenberger2019robust}, cf. also \cite{ferizbegovic2019learning,iannelli2019structured}, attempts to do `targeted-exploration' by prioritizing uncertainty reduction in the system parameters to which performance is most sensitive. 
These methods consider worst-case costs to bound performance on the true system, and as such, can be conservative in practice.
It is the ambition of this paper the benefits of `targeted exploration' with a more `optimistic' CE strategy, that optimizes for performance of the nominal, rather than worst-case, system.
Other recent work on adaptive linear quadratic control includes Thompson sampling (e.g. \cite{ouyang2017learning,abeille2017thompson,abeille2018improved}), model-free (e.g. \cite{fazel2018global,malik2018derivative}) and partially model-free (e.g. \cite{agarwal2019online,agarwal2019logarithmic}) methods, as well as the `optimism in the face of uncertainty' heuristic (e.g. \cite{abbasi2011regret,ibrahimi2012efficient,faradonbeh2017finite}).

\section{Problem statement}\label{sec:problem}

In this section we describe in detail the problem addressed in this paper.
Notation is largely standard.
$\kron$ denotes the Kronecker product.
$\sym{n}$ denotes the space of $n\times n$ symmetric matrices.
w.p. means `with probability'.
$\chiSquare{n}{p}$ denotes the value of the Chi-squared distribution with $n$ degrees of freedom and probability $p$.
The space of real, proper (strictly proper) transfer matrices is denoted $\rhinf$ ($\frac{1}{z}\rhinf$).
With some abuse of notation, $[t_1,t_2]$ for $t_1,t_2\in\naturals$ denotes $\lbrace t_1,\dots,t_2\rbrace$.

\paragraph{Dynamics and modeling}
We are concerned with control of linear time-invariant systems
\begin{align}
x_{t+1} = A x_t + B u_t + w_t, \quad
w_t \sim \normal{0}{\diststd^2I_{n_x}}, \quad x_0 = 0,
\label{eq:lti}
\end{align}
where $x_t \in \real^{n_x}$, $u_t \in \real^{n_u}$ and $w_t \in \real^n$ denote the state (which is assumed to be directly measurable), input and process noise, respectively, at time $t$. 
We assume that the true parameters $\lbr \atr,\btr\rbr$ are unknown; as such, all knowledge about the true system dynamics must be inferred from observed data, 
$\data_n \defeq \lbrace x_t,u_t\rbrace_{t=1}^n$.
We assume that $\diststd$ is known, or has been estimated, and that we have access to initial data, denoted (with slight notational abuse) $\datainit$, obtained, e.g. during a preliminary experiment. 
Given data $\data_n$ we define a model $\model(\data_n) = \lbr \anom,\bnom,\uncert\rbr$,
where $(\anom,\bnom) \defeq \arg\min_{A,B} \sum_{t=1}^{n-1}|x_{t+1}-Ax_t-Bu_t|^2$ denote \emph{nominal} parameters given by the ordinary least squares estimates of $(\atr,\btr)$, and $\uncert\in\sym{n_x+n_u}$ is a matrix that quantifies the uncertainty in our nominal parameter estimate.
Specifically, given a user-specified tolerance $0<\credprob<1$, $\uncert \defeq \frac{1}{\diststd^2\credconst}\sum_{t=1}^{n-1}
\left[\begin{array}{c}
	x_t \\ u_t
\end{array}\right]
\left[\begin{array}{c}
	x_t \\ u_t
\end{array}\right]\trnsp$,
where $\credconst = \chiSquare{n_x^2+n_xn_u}{\credprob}$, defines a $1-\credprob$ probability credibility region for the true parameters as follows:

\begin{lemma}[\cite{umenberger2019robust}]\label{lem:credibility_region}
Given data $\data_n$ from \eqref{eq:lti} with true parameters $A=\atr$ and $B=\btr$, and a user-specified $0<\credprob<1$, define the set
\begin{equation}\label{eq:spectral_region}
\matrixmodelset(\model(\data_n)) \defeq \lbr A, \ B \ : \ \modelx\trnsp\uncert \modelx \mle I, \ \modelx = [\anom - A, \ \bnom - B]\trnsp \rbr.
\end{equation}
where $\lbr \anom,\bnom,\uncert\rbr = \model(\data_n)$. 
Then
$\lbr\atr, \ \btr\rbr\in \matrixmodelset(\model)$ w.p. $1-\credprob$.
\end{lemma}

Lemma \ref{lem:credibility_region} is a consequence of the fact the posterior distribution of parameters $A,B$ (for a uniform prior) is Gaussian for models of the form \eqref{eq:lti}, cf. e.g. \cite{umenberger2018learning}.
Similar credibility regions have been attained using results from high-dimensional statistics in recent works such as \cite{dean2017sample}.

\paragraph{Control objective}
Our objective is to design a feedback control policy 
$u_t = \genpolicy( x_{1:t},u_{1:t-1})$
so as to minimize the cost function 
$\sum_{t=1}^\totaltime \ \cost(x_t,u_t)$, where $\cost(x_t,u_t) = x_t\trnsp\qcost x_t + u_t\trnsp\rcost u_t $ for user-specified positive semidefinite matrices $\qcost$ and $\rcost$.
When the parameters of the true system, $\lbr \atr,\btr\rbr$, are known this is the well-known LQR problem.
As discussed, we do not assume knowledge of the true parameters; as such, the controller must regulate and learn the system simultaneously.  
To this end, we partition the total `control time' $[1,\totaltime]$ into two intervals: $[1,\timexp]$ and $[\timexp+1,\totaltime]$ for $\timexp\in\naturals$.
At time $t=1$, given initial data $\data_0$, a policy $\gpoli$ is designed and applied to the system for $t\in[1,\timexp]$.
Then, at time $t=\timexp+1$, a new policy $\gpolii$ is designed, based on $\data_0$ and data $\data_{\timexp}$ collected under $\gpoli$.
Policy $\gpolii$ is then applied for $t\in[\timexp+1,\totaltime]$.
We can write this control task as:
	\begin{align}\label{eq:ideal_control_task}
	\min_{\gpoli,\gpolii} \quad   \ev {\sum}_{t=1}^\totaltime \cost(x_t,u_t), \quad
	\st \  &\textup{ dynamics in } \eqref{eq:lti} \text{ with } A=\atr, \ B=\btr, \\
	& u_t = \gpoli(\cdot), \ t = 1,\dots,\timexp, \
	u_t = \gpolii(\cdot), \ t = \timexp+1,\dots,\totaltime. \nonumber
	\end{align}
One can think of the first interval, $[1,\timexp]$, as an `exploration' or `learning' period where the data collected is used to design an improved controller, $\gpolii$, applied during the second interval $[\timexp+1,\totaltime]$, which could be considered an `exploitation' period.
However, the task is to minimize the total cost 
therefore, it is important to balance exploration and exploitation.
The decision to nominate a specific time, $\timexp$, at which the control policy will be `updated' 
requires some justification.
A more natural formulation might update the controller whenever new data becomes available.
We shall discuss this aspect of the formulation in more detail in \mysec\ref{sec:discussion}, where alternative formulations are considered.
For now, suffice to say that this formulation, i) simplifies the presentation of the technical developments to follow, and ii) still captures the importance of balancing `exploration' with `exploitation.'

Observe that the control task in \eqref{eq:ideal_control_task} depends on the true, but unknown, system parameters $\atr,\btr$, as we want to optimize for performance on the true system.
In place of the true system parameters, we will optimize for our `best guess' of the parameters, i.e., the nominal parameters $\anom,\bnom$ 
from least squares corresponding to the mode of the posterior distribution.
To ensure reasonable behavior of the true system in closed-loop, we also require the controllers to stabilize the true system with high probability.
Let $\cloop(A,B,\genpolicy)$ denote the closed loop system formed by combining \eqref{eq:lti} with the policy $\genpolicy$.
The problem addressed in this paper as follows:
\begin{subequations}\label{eq:control_task}
	\begin{align}
	\min_{\gpoli,\gpolii} & \quad  \ev {\sum}_{t=1}^\totaltime \cost(x_t,u_t) \\
	\st \  x_{t+1} &= \anom_1 x_t + \bnom_1 u_t + w_t, \ \lbr \anom_1,\bnom_1\rbr = \model(\data_0), \ u_t = \gpoli(\cdot),\ t\in[1,\timexp] \\
	x_{t+1} &= \anom_2 x_t + \bnom_2 u_t + w_t, \ \lbr \anom_2,\bnom_2\rbr = \model(\data_0\cup\data_\timexp), \ u_t = \gpolii(\cdot),\ t\in [\timexp+1,\totaltime], \\
	 & \cloop(A,B,\gpoli) \textup{ is stable } \forall \lbr A,B\rbr \in \matrixmodelset(\model(\data_0)) \label{eq:stability_1} \\
	 & \cloop(A,B,\gpolii) \textup{ is stable } \forall \lbr A,B\rbr \in \matrixmodelset(\model(\data_0\cup\data_\timexp)), \ w_t \sim \normal{0}{\diststd^2I_{n_x}} \forall t.
	\end{align}
\end{subequations}

\section{Controller synthesis}\label{sec:control_synthesis}

In what follows, \textbf{bold} symbols denote the z-transform of time domain signals, e.g., the z-transform of $x$ is denoted $\xzt$.

\subsection{Preliminary results from System Level Synthesis}\label{sec:sls}
In this section we review some essential results from the System Level Synthesis (SLS) framework proposed by \cite{wang2019system}; for a comprehensive tutorial, cf. \cite{anderson2019system}.
Consider the closed-loop behavior of \eqref{eq:lti} under the stabilizing controller $\uzt = \kzt\xzt$; in particular, consider the transfer functions $\xrespzt$ and $\urespzt$ from disturbance $\wzt$ to state $\xzt$ and control $\uzt$, respectively. 
This can be expressed as
$
\left[\begin{array}{c}
\xzt \\ \uzt
\end{array}\right]
=
\left[\begin{array}{c}
\xrespzt \\ \urespzt
\end{array}\right]
\wzt.
$
Following in the spirit of the Youla parameterization, rather than designing the controller $\kzt$ to obtain the closed-loop responses  $\xrespzt = (zI-A-B\kzt)^{-1}$ and $\urespzt=\kzt\xrespzt$, in SLS one designs the closed-loop responses directly, and then recovers the controller as $\kzt = \urespzt\xrespzt^{-1}$.
The following theorem characterizes the space of all closed-loop responses achievable by a stabilizing controller.
\begin{theorem}[\cite{anderson2019system}, Theorem 4.1]\label{thm:sls_affine}
The affine subspace defined by 
\begin{equation}\label{eq:sls_affine}
\left[\begin{array}{cc}
zI - A & -B
\end{array}\right]
\left[\begin{array}{c}
\xrespzt \\ \urespzt
\end{array}\right]
= I, \quad \xrespzt,\urespzt\in\frac{1}{z}\rhinf,
\end{equation}
parametrizes all closed-loop responses achievable by a stabilizing controller.
Further, the response is achieved by the controller $\kzt = \urespzt\xrespzt^{-1}$.
\end{theorem}
The following theorem considers a `perturbed' version of the constraints in \eqref{eq:sls_affine}, that is useful for synthesizing robust controllers, e.g., when the system parameters $A$,$B$ are uncertain.
\begin{theorem}[\cite{anderson2019system}, Theorem 4.3]\label{thm:sls_robust}
Suppose that $\xrespzt,\urespzt,\Deltazt$ satisfy 
	\begin{equation}\label{eq:sls_robust}
	\left[\begin{array}{cc}
	zI - A & -B
	\end{array}\right]
	\left[\begin{array}{c}
	\xrespzt \\ \urespzt
	\end{array}\right]
	= I + \Deltazt, \quad \xrespzt,\urespzt\in\frac{1}{z}\rhinf.
	\end{equation}
	Then the controller $\kzt = \urespzt\xrespzt^{-1}$ stabilizes system \eqref{eq:lti} with parameters $A,B$ if and only if $(I+\Deltazt)^{-1}\in\rhinf$.
\end{theorem}
While the preceding theorems define affine subspaces (i.e. \eqref{eq:sls_affine} and \eqref{eq:sls_robust}) that are convenient to optimize over, the decision variables $\xrespzt$ and $\urespzt$ are infinite dimensional transfer matrices.
As is common in the SLS framework, we will work with finite impulse response (FIR) approximations:
\begin{equation}\label{eq:fir_def}
\xfirzt(z) = {\sum}_{k=0}^\firl \xresp^kz^{-k}, \quad \ufirzt(z) = {\sum}_{k=0}^\firl \uresp^kz^{-k}.
\end{equation}
Henceforth, we will restrict our attention to policies of the form $\genpolicy(z) = \ufirzt(z)\xfirzt(z)^{-1}$.

\subsection{Robust control formulation}\label{sec:robust_synthesis}
In this section, we present a convex formulation of the following problem:
(approximately) optimize the infinite-horizon quadratic cost, for a given nominal model $\lbr\anom,\bnom\rbr$, while robustly stabilizing all models $\lbr A,B\rbr$ in the model set $\mset(\model)$.
This result extends existing SLS formulations, by preserving the structure in the representation of uncertainty captured by $\uncert$.

Following straightforward calculations, cf. e.g. \cite[\mysec2.2.2]{anderson2019system}, the infinite horizon cost function can be written as
\begin{equation}\label{eq:h2_cost}
\lim_{\tau\rightarrow\infty}\frac{1}{\tau}\sum_{t=1}^\tau \ev[\cost(x_t,u_t)] 
= \ev \norm{\left[\begin{array}{cc}
	Q^{\frac{1}{2}} & 0 \\ 0 & R^{\frac{1}{2}}
	\end{array}\right]
	\left[\begin{array}{c}
	\xzt \\ \uzt
	\end{array}\right]}{F}^2
= \diststd^2 \norm{\left[\begin{array}{cc}
	Q^{\frac{1}{2}} & 0 \\ 0 & R^{\frac{1}{2}}
	\end{array}\right]
	\left[\begin{array}{c}
	\xfirzt \\ \ufirzt
	\end{array}\right]}{\htwo}^2
\end{equation}
subject to the affine constraints in \eqref{eq:sls_affine} with the nominal parameters $\anom,\bnom$, i.e., 
\begin{equation}\label{eq:affine_nominal}
\left[\begin{array}{cc}
zI - \anom & -\bnom
\end{array}\right]
\left[\begin{array}{c}
\xrespzt \\ \urespzt
\end{array}\right]
= I.
\end{equation}
By making use of the FIR approximation in \eqref{eq:fir_def}, the rightmost side of \eqref{eq:h2_cost} can be written as:
\begin{equation}\label{eq:h2_objective}
\htwocost = \diststd^2 \norm{\left[\begin{array}{cc}
	Q^{\frac{1}{2}}\kron I_F & 0 \\ 0 & R^{\frac{1}{2}}\kron I_F
	\end{array}\right]
	\left[\begin{array}{c}
	\xfirstk \\ \ufirstk
	\end{array}\right]}{F}^2,
\end{equation}
where $\xfirstk = \left[ \xresp\trnsp(0) \dots \xresp\trnsp(\firl)\right]\trnsp$ and 
$\ufirstk = \left[ \uresp\trnsp(0) \dots \uresp\trnsp(\firl)\right]\trnsp$, denote the FIR parameters from \eqref{eq:fir_def}, stacked vertically.
Note that \eqref{eq:h2_objective} is a convex quadratic function of $\xfirstk$ and $\ufirstk$.

To ensure robustness of the policy on the true, unknown system (as in, e.g., \eqref{eq:stability_1}), we make use of Theorem \ref{thm:sls_robust}.
Specifically, as $\xfirzt$ and $\ufirzt$ are constrained to satisfy \eqref{eq:affine_nominal}, we can express $\Deltazt$ in \eqref{eq:sls_robust} as 
$\Deltazt = (\anom-\atr)\xfirzt + (\bnom-\btr)\ufirzt$,
by substituting \eqref{eq:affine_nominal} into \eqref{eq:sls_robust}, with $A=\atr$ and $B=\btr$.
By Theorem \ref{thm:sls_robust}, the controller $\ufirzt\xfirzt^{-1}$ will stabilize the true system, if and only if $(I+\Deltazt)^{-1}$ is stable.
Of course, $\Deltazt$ is defined in terms of $\atr,\btr$, which are unknown; however, by Lemma \ref{lem:credibility_region}, they are known to lie in $\mset(\model)$ with high-probability.
A sufficient condition for stability of $(I+\Deltazt)^{-1}$ is given by the small gain theorem: 
$\norm{\Deltazt}{\hinf}\leq1$ implies stability of $(I+\Deltazt)^{-1}$.
The 
$\hinf$-norm
of a transfer matrix can be computed/constrained as follows:
\begin{lemma}[\cite{dumitrescu2007positive}]\label{lem:hinf}
Let $\dumtfzt(z)=\sum_{i=0}^\firl\dumtf(i)z^{-i}$, with $\dumtf\in\real^{p\times m}$ and $\bar{\dumtf}=$ $\left[ \dumtf(0)\trnsp \dots \dumtf(\firl)\trnsp \right]\trnsp$. 
Then $\norm{\dumtfzt}{\hinf}\leq \gamma$ iff there exists $\cert\in\sym{p(\firl+1)}$ satisfying 
\begin{subequations}
\begin{align}
\cert = &\left[\begin{array}{cccc}
\cert_{00} & \cert_{01} & \cdots & \cert_{0\firl} \\
\star & \cert_{11} & \cdots & \cert_{1\firl} \\
\star & \star & \ddots & \vdots \\
\star & \star & \star & \cert_{\firl\firl} 
\end{array} \right],
\sum_{i=0}^\firl \cert_{ii} = \gamma I, 
\sum_{i=0}^{\firl-k}\cert_{i(i+k)}=0, k=1,\dots,\firl, \label{eq:hinf_bound_cond} \\
&\left[ \begin{array}{cc}
\cert & \bar{\dumtf} \\ \bar{\dumtf}\trnsp & I
\end{array} \right] \mge 0. \label{eq:hinf_lmi}
\end{align}
\end{subequations}
\end{lemma}
We can now present the main contribution of this section.


\begin{theorem}\label{thm:robust_synth}
	Given a model $\model = \lbr \anom,\bnom,\uncert\rbr$, a convex upper bound for  $\min_\genpolicy\lim_{\tau\rightarrow\infty}\frac{1}{\tau}\sum_{t=1}^\tau \ev[\cost(x_t,u_t)]$ for a system \eqref{eq:lti} with parameters $A=\anom$ and $B=\bnom$, subject to the constraint that the controller $\genpolicy$ stabilizes all models in $\mset(\model)$, is given by
	$\robustcostopt(\model) = \min_\genpolicy \robustcost(\genpolicy,\model)$, 		
	where

	\begin{align}\label{eq:robust_synth}
	\robustcost(\genpolicy,\model) \defeq \Bigg\lbrace \htwocost(\xfirzt,\ufirzt) \mid 
	\eqref{eq:affine_nominal}, \ \exists \cert, \multiplier\in\real\ \st \ \eqref{eq:hinf_bound_cond},  
	\left[\begin{array}{ccc}
	\cert & 0 & \xuresp \\ 
	0 & (1-\multiplier)I & 0 \\
	\xuresp\trnsp & 0 & \multiplier\uncert
	\end{array}\right] \mge 0 	\Bigg\rbrace
	\end{align}
	where $\xuresp=\left[\begin{array}{cccc}
	\xresp(0) & \dots & \xresp(\firl) \\
	\uresp(0) & \dots & \uresp(\firl) \\
	\end{array}\right]\trnsp$,
	and $\genpolicy$ can be realized as $\genpolicy = \ufirzt\xfirzt^{-1}$.
\end{theorem}

\begin{proof}
The convex program $\min_{\xfirzt,\ufirzt,\cert} \ \htwocost(\xfirzt,\ufirzt)$ s.t. \eqref{eq:affine_nominal} optimizes the infinite-horizon cost.
All that remains is to ensure robust stability: a sufficient condition is $\norm{\Deltazt}{\hinf}\leq 1$.
This sufficient, but not necessary, condition is the source of the conservatism, i.e., the reason we only have an upper bound.
$\norm{\Deltazt}{\hinf}\leq 1$ can be enforced by combining Lemma \ref{lem:hinf} with the following lemma:
\begin{lemma}[\cite{luo2004multivariate}]\label{thm:robustlmi}
	The data matrices $(\lmia,\lmib,\lmic,\lmid,\lmif,\lmig,\lmih)$ satisfy,
	for all $X$ with $I-X\trnsp\lmid X\mge 0$,
	the robust fractional quadratic matrix inequality
	\begin{align}\label{eq:robust_LMI}
	\left[\begin{array}{cc}
	\lmih & \lmif+\lmig X \\
	(\lmif+\lmig X)\trnsp & \lmic+X\trnsp \lmib + \lmib \trnsp X +X\trnsp \lmia X \end{array}\right] \mge 0,
	\textup{ iff }
	\left[\begin{array}{ccc}
	\lmih & \lmif & \lmig\\
	\lmif\trnsp & \lmic - \multiplier I & \lmib\trnsp \\\lmig\trnsp & \lmib & \lmia + \multiplier\lmid  \end{array}\right]\mge 0,
	\end{align}
	for some $\multiplier \geq 0$.
\end{lemma}
Note that $\Deltazt\trnsp = [\xfirzt\trnsp, \ \ufirzt\trnsp ][ \anom - \atr, \ \bnom - \btr ]\trnsp
= [\xfirzt\trnsp, \ \ufirzt\trnsp ]\modelx$ where $\modelx$ is defined in \eqref{eq:spectral_region}.
Furthermore, $\norm{\Deltazt}{\hinf}\leq1 \iff \norm{\Deltazt\trnsp}{\hinf}\leq1$.
Let $\dumtfzt(z)$ (from Lemma \ref{lem:hinf}) denote $\Deltazt\trnsp$,
then the \eqref{eq:hinf_lmi} can be put in the form of the MI on the left in \eqref{eq:robust_LMI} by choosing
$\lmih=\cert$,
$\lmig=\xuresp$,
$\lmic=I$,
and $\lmia,\lmib,\lmif$ all zero.
Further, by choosing $\lmid=\uncert$ in Lemma \ref{thm:robustlmi}, the condition $I\mge X\trnsp\lmid\modelx$ is equivalent to $\lbr \atr,\btr\rbr\in\mset(\model)$, cf. \eqref{eq:spectral_region}.
The final constraint (LMI) in \eqref{eq:robust_synth} is then equivalent to the second LMI in \eqref{eq:robust_LMI}, which implies that $\norm{\Deltazt\trnsp}{\hinf}\leq 1$ for the true model parameters w.p. $1-\delta$.
\end{proof}
Observe that this formulation preserves the structure in the uncertainty representation, as encoded in the matrix $\uncert$.
This is in contrast to other SLS methods, e.g., \cite{dean2017sample,dean2018safely}, that reduce model uncertainty to a single (or at most two) scalar quantities, e.g., 
$\Vert\anom-\atr\Vert_2\leq \epsilon_A$.

\subsection{Robust dual control formulation}\label{sec:dual_synthesis}
In this section we return to the dual control problem outlined in \eqref{eq:control_task}, namely: minimize cost over $[1,\totaltime]$, via an initial policy $\gpoli$, designed using data $\data_0$, and applied for $t\in[1,\timexp]$, followed by a second policy $\gpolii$, designed using additional data $\data_\timexp$, and applied for $t\in[\timexp+1,\totaltime]$.
The key idea is dual control: $\gpoli$ affects not only the cost, but also the data $\data_\timexp$ available for the design of $\gpolii$.

\paragraph{Infinite-horizon approximation}
In what follows, we will approximate the cost $\sum_{t=1}^\timeh \ev c(x_t,u_t)$ by the infinite-horizon (i.e. stationary value) 
$\timeh\times\lim_{\tau\rightarrow\infty}\frac{1}{\tau}\sum_{t=1}^\tau\ev c(x_t,u_t)$.
Such an approximation can be expected to be valid when the horizon $\timeh$ is sufficiently long, so as to allow the system to reach the stationary distribution.
Alternatives to this approximation are discussed in \mysec\ref{sec:discussion}.
With the infinite horizon approximation, we can express \eqref{eq:control_task} as 
\begin{equation}\label{eq:compact_dual}
\min_{\gpoli} \ \timexp\times\robustcost(\gpoli,\model(\data_0)) + (\totaltime-\timexp)\times \robustcostopt(\model(\data_0\cup\data_\timexp)).
\end{equation}
Note the dependence of the cost during $[\timexp,\totaltime]$ on $\data_{\timexp}$. 
Note also that $\gpolii$ is defined implicitly by $\robustcostopt$, cf. \eqref{eq:robust_synth}.
Problem \eqref{eq:compact_dual} cannot be solved exactly, as it depends on $\data_{\timexp}$ which is not available at time $t=1$.
As such, we must predict the influence that $\gpoli$ will have on `future' data $\data_{\timexp}$.

\paragraph{Propagating uncertainty}
To approximately solve \eqref{eq:compact_dual}, we require an approximation of $\model(\data_0\cup\data_\timexp)$,
i.e., $\amodel(\gpoli) \approx \ev[ \model(\data_0\cup\data_{\timexp}) | \data_0,\gpoli]$.
Let $\model(\data_0)=\lbr \anom_1,\bnom_1,\uncert_1\rbr$.
We then define the approximate model as $\amodel(\gpoli)\eqdef \lbr \anom_1,\bnom_1,\auncert\rbr$.
First, note that we approximate the predicted nominal parameters by the current estimates.
Updating these estimates based on the expected value of future data involves difficult integrals that must be computed numerically, which would destroy convexity, cf. \cite{lobo1999policies}.
The predicted `uncertainty matrix' $\auncert$ is defined as follows.
Given $\data_\timexp$, $\uncert$ at time $\timexp$ can be computed as
$\uncert_1 + 
\frac{1}{\diststd^2\credconst}
\sum_{t=1}^{\timexp}\left[ \begin{array}{c}
x_t \\ u_t
\end{array}\right]
\left[ \begin{array}{c}
x_t \\ u_t
\end{array}\right]\trnsp$.
We can approximate the empirical covariance with stationary distribution over $x$ and $u$; i.e., as with the stationary approximation of the cost in \eqref{eq:h2_cost}, we have:
$
\sum_{t=1}^{\timexp}\left[ \begin{array}{c}
x_t \\ u_t
\end{array}\right]
\left[ \begin{array}{c}
x_t \\ u_t
\end{array}\right]\trnsp
\approx
\timexp\ev
\left[ \begin{array}{c}
x_t \\ u_t
\end{array}\right]
\left[ \begin{array}{c}
x_t \\ u_t
\end{array}\right]\trnsp
=
\timexp\diststd^2 \xuresp\trnsp\xuresp.
$
We then define $\auncert\defeq\uncert_1 + \frac{\timexp}{\credconst}\xuresp\trnsp\xuresp$.

\paragraph{Convex relaxation of dual control}
By substituting the approximate model $\amodel(\gpoli)$ for $\model(\data_0\cup\data_{\timexp})$ in \eqref{eq:compact_dual}, we can remove the dependence of the cost on unknown future data.
Furthermore, observe that (for fixed $\multiplier$) the LMI constraint in \eqref{eq:robust_synth} is linear in $\uncert$.
This implies that we can optimize over $\uncert$ and $\genpolicy$ jointly, as a convex program.
Unfortunately, $\auncert$ in $\amodel(\gpoli)$ is a quadratic function of $\gpoli$, and so directly substituting $\uncert\rightarrow\auncert$ results in a non-convex matrix inequality.

To circumvent this difficulty, we introduce the following linear approximation of $\auncert$:
\begin{equation}
\uncertlin \defeq \uncert_1 + \frac{\timexp}{\credconst}\left( \xuresp\trnsp\xulp + \xulp\trnsp\xuresp - \xulp\trnsp\xulp\right), \quad \xulp=\arg\min_\policy \robustcost(\policy,\model(\data_0)).
\end{equation}
$\uncertlin$ is nothing more than a first-order approximation of $\auncert$, linearized at $\xulp$.
We could linearize at any arbitrary point; however, the solution $\xulp$ to the nominal robust control problem given $\data_0$ is a natural choice.
Furthermore, $\uncertlin\mle\auncert$, as $(\xuresp-\xulp)\trnsp(\xuresp-\xulp)\mge0\iff \xuresp\trnsp\xuresp\mge\xuresp\trnsp\xulp + \xulp\trnsp\xuresp - \xulp\trnsp\xulp$.
We now present the main contribution of this section:
given the initial model $\model(\data_0)=\lbr \anom_1,\bnom_1,\uncert_1\rbr$,
the problem 
$\min_{\gpoli} \ \timexp\times\robustcost(\gpoli,\model(\data_0)) + (\totaltime-\timexp)\times \robustcostopt(\amodel(\gpoli))$
admits the following convex upper bound:
\begin{subequations}\label{eq:final_convex_prog}
\begin{align}
& \min_{\xfirzt^1,\ufirzt^1,\xfirzt^2,\ufirzt^2,\cert^1,\cert^2,\lambda^1} \  \timexp\times\htwocost(\xfirzt^1,\ufirzt^1) + (\totaltime-\timexp)\times \htwocost(\xfirzt^2,\ufirzt^2) \\
\st \quad & \lbr \xfirzt^1,\ufirzt^1\rbr \textup{ and }  \lbr\xfirzt^2,\ufirzt^2\rbr \textup{ each satisfy } \eqref{eq:affine_nominal} \textup{ with } A=\anom_1, B=\bnom_1 \\
& \cert^1, \cert^2 \textup{ each satisfy } \eqref{eq:hinf_bound_cond} \\
& \left[\begin{array}{ccc}
\cert^1 & 0 & \xuresp^1 \\ 
0 & (1-\multiplier^1)I & 0 \\
(\xuresp^1)\trnsp & 0 & \multiplier^1\uncert_1
\end{array}\right] \mge 0,
\left[\begin{array}{ccc}
\cert^2 & 0 & \xuresp^2 \\ 
0 & (1-\multiplier^2)I & 0 \\
(\xuresp^2)\trnsp & 0 & \multiplier^2\uncertlin
\end{array}\right] \mge 0\label{eq:final_lmi}
\end{align}
\end{subequations}
where $\xuresp^{i}$ is defined analogously to $\xuresp$ in Theorem \ref{thm:robust_synth}, for $i=1,2$.
As $\uncertlin\mle\auncert$, the feasible set for \eqref{eq:final_convex_prog} with $\uncertlin$ (in \eqref{eq:final_lmi}) is smaller than that of the program with (quadratic) $\auncert$. Therefore, \eqref{eq:final_convex_prog} constitutes an upper bound.
Notice that \eqref{eq:final_convex_prog} is only convex for fixed $\multiplier^2$, due to bilinearity with $\uncertlin$.
In practice, one has to grid search over the scalar parameter $\multiplier_2$.

\subsection{Discussion}\label{sec:discussion}
In this section, we discuss the partitioning of $[1,\totaltime]$ into two sub-intervals.
A downside of this approach is that it requires the user to explicitly specify the `exploration' period, i.e., select $\timexp$.
The proposed approach could also be considered a `one-step-look-ahead' dual control, as there is only a single period of exploration ($[1,\timexp]$), before a single period of exploitation ($[\timexp,\totaltime]$).
Such a drawback can be partially mitigated by adopting a 'multistep-look-ahead' strategy, as in \cite{umenberger2019robust}. 
In such a framework, the current period of exploration is followed by further exploration, rather than pure exploitation.
This approach also requires the user to select `epoch times', at which the controller will be updated.
To circumvent this, one could adopt a model predictive control (MPC) strategy, with SLS as in \cite{wang2019robust}, but exploiting the dual control effect, as in \cite{lobo1999policies}.
The major obstacle to extending the proposed approach to the multistep or MPC setting is the need to search over more than one multiplier, $\multiplier_2$, cf. \eqref{eq:final_lmi}.
Such extensions represent interesting directions for future research.
At any rate, the method proposed in this paper could be used (with fixed multipliers) as a convex means of providing sub-optimal initialization for local search methods.

\section{Numerical illustration}\label{sec:example}
In this section we illustrate the proposed method with a numerical example.
Consider the linear control problem with parameters:
\begin{equation*}
\atr = \left[ \begin{array}{cc}
0.5 & 1.1 \\ 0 & 0.8
\end{array} \right], \
\btr = \left[ \begin{array}{cc}
1 & 0 \\ 0 & 1
\end{array} \right], \
\qcost= \left[ \begin{array}{cc}
1 & 0 \\ 0 & 0.001
\end{array} \right], \
\qcost= \left[ \begin{array}{cc}
10^3 & 0 \\ 0 & 10^3
\end{array} \right], \
\diststd = 1, \credprob=0.1
\end{equation*}
Initial data $\data_0$ is generated by simulating \eqref{eq:lti} open-loop with $u_t\sim\normal{0}{I}$ for $t=6$ timesteps; this is repeated 10 times.
For control, $\totaltime=100$, $\timexp$, and $\firl=12$.
We compare three methods:
i) nominal control: $\gpoli=\arg\min_\genpolicy\robustcost(\genpolicy,\model(\data_0))$ and $\gpolii=\arg\min_\genpolicy\robustcost(\genpolicy,\model(\data_0\cup\data_{\timexp}))$, i.e., no explicit exploration;
ii) dual control: as proposed in this paper;
iii) greedy control: $\gpoli$ is given by $u=\kzt x+e$ where $\kzt$ is the nominal controller (i), and $e\sim\normal{0}{\sigma I}$. $\sigma$ is tuned to give the same exploration cost as dual control on the true, and represents a `greedy strategy' of injecting as much `exploration signal' $e$ into the input as possible, as opposed to targeting uncertainty reduction in specific parameters. $\gpolii$ is given by the `nominal' control. This experiment is repeated 1000 times, with the results presented in Fig. \ref{fig:main_res}.
The greedy strategy performs slightly better during exploitation, but this cannot offset the cost of exploration, leading to worse performance in terms of total cost.
Dual control balances exploration and exploitation to achieve the lowest total cost.	

\begin{figure}[H]
	\centering
	\includegraphics[width=0.9\linewidth]{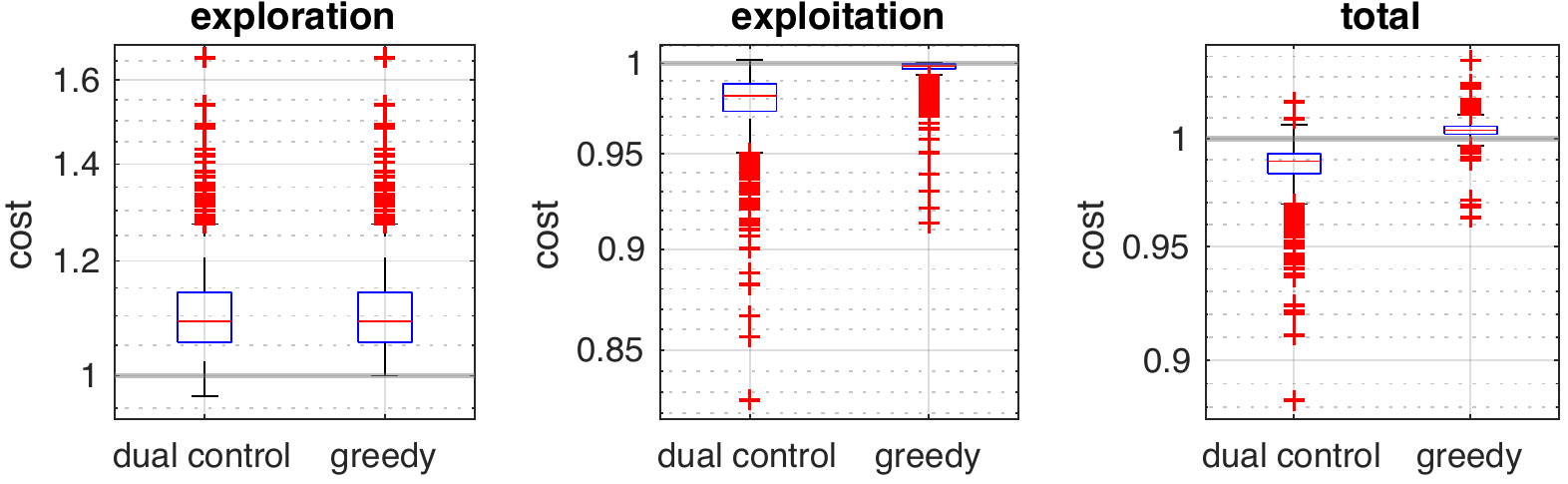}
	
	\caption{Costs during exploration ($t\in[1,\timexp]$), exploitation ($t\in[\timexp,\totaltime]$), and the total cost (exploration $+$ exploitation). Costs are normalized by the cost of the nominal control (i.e. unity implies the same cost as the nominal control). Dual control exhibits best performance.
	}
	\label{fig:main_res}
\end{figure}


\acks{This research was financially supported by the Swedish Foundation for Strategic Research (SSF) via the project \emph{ASSEMBLE} (contract number: RIT15-0012) and by the project \emph{NewLEADS - New Directions in Learning Dynamical Systems} (contract number: 621-2016-06079), funded by the Swedish Research Council.}

\end{document}